\documentclass{amsart}
\pdfoutput=1
\usepackage{latexsym}
\usepackage{amssymb, latexsym}
\usepackage{amsmath}
\usepackage{mathrsfs}
\usepackage{amsxtra}
\usepackage{diagrams}
\usepackage{amsthm}
\usepackage[hyphens]{url}
\newtheorem{theorem}{Theorem}[section]

\newtheorem*{main1}{Main~Theorem~1}
\newtheorem*{main2}{Main~Theorem~2}
\theoremstyle{definition}
\newtheorem{definition}[theorem]{Definition}

\newtheorem{question}[theorem]{Question}

\newcommand{\plusplus}{{{+}{+}}}
\newcommand{\ZF}{{\rm ZF}}
\newcommand{\DC}{{\rm DC}}
\newcommand{\AC}{{\rm AC}}

\newcommand{\image}{\mathbin{\hbox{\tt\char'42}}}
\newcommand{\plus}{{+}}

\newcommand{\Union}{\bigcup}

\newcommand{\of}{\subseteq}
\newcommand{\lt}[1]{{\smalllt}#1}
\newcommand{\lesseq}[1]{{\smallleq}#1}
\newcommand{\smallleq}{\mathrel{\mathchoice{\raise2pt\hbox{$\scriptstyle\leq$}}{\raise1pt\hbox{$\scriptstyle\leq$}}{\raise1pt\hbox{$\scriptscriptstyle\leq$}}{\scriptscriptstyle\leq}}}
\newcommand{\smalllt}{\mathrel{\mathchoice{\raise2pt\hbox{$\scriptstyle<$}}{\raise1pt\hbox{$\scriptstyle<$}}{\raise0pt\hbox{$\scriptscriptstyle<$}}{\scriptscriptstyle<}}}
\newcommand{\Add}{\mathop{\rm Add}}

\newcommand{\GCH}{{\rm GCH}}
\newcommand{\Z}{{\rm Z}^*}
\newcommand{\ORD}{\mathop{{\rm ORD}}}
\newcommand{\ZFC}{{\rm ZFC}}
\newcommand{\Levy}{L{\'e}vy}
\newcommand{\one}{\mathop{1\hskip-2.5pt {\rm l}}}

\newcommand{\p}{\mathbb{P}}
\newcommand{\q}{\mathbb{Q}}
\newcommand{\R}{\mathbb{R}}

\newcommand{\la}{\langle}
\newcommand{\ra}{\rangle}

\newcommand{\her}[1]{H_{{#1}^+}}

\newcommand{\forces}{\Vdash}

\newcommand{\Los}{\L o\'s}

\newcommand{\restrict}{\upharpoonright}

\newarrow{Corresponds}<--->
\newarrow{Dashto}{}{dash}{}{dash}>

\newcommand{\intersect}{\cap}
\def\<#1>{\langle#1\rangle}
\newcommand{\st}{\mid}

\newcommand{\dom}{\text{dom}}

\newcommand{\Adot}{\dot{A}}
\begin{document}
\title{On ground model definability}
\author{Victoria Gitman and Thomas A. Johnstone}
\maketitle
\begin{abstract}
Laver, and Woodin independently, showed that models of $\ZFC$ are uniformly definable in their set-forcing extensions, using a ground model parameter \cite{laver:groundmodel,woodin:groundmodel}. We investigate ground model definability for models of fragments of $\ZFC$, particularly of $\ZF+\DC_\delta$ and of $\ZFC^-$, and we obtain both positive and negative results. Generalizing the results of \cite{laver:groundmodel}, we show that models of $\ZF+\DC_\delta$ are uniformly definable in their set-forcing extensions by posets admitting a gap at~$\delta$, using a ground model parameter.  In particular, this means that  models of $\ZF+\DC_\delta$ are uniformly definable in their forcing extensions by  posets of size less than $\delta$. We also show that it is consistent for ground model definability to fail for canonical $\ZFC^-$ models $\her{\kappa}$. Using forcing, we produce a $\ZFC$ universe in which there is a cardinal $\kappa>\!>\omega$ such that $\her{\kappa}$  is not definable in its Cohen forcing extension. As a corollary, we show that there is always a countable transitive model of $\ZFC^-$ violating ground model definability. These results turn out to have a bearing on ground model definability for models of $\ZFC$.  It follows from our proof methods that the hereditary size of the parameter that Woodin used in \cite{woodin:groundmodel} to define a $\ZFC$ model in its set-forcing extension is best possible.
\end{abstract}
\section{Introduction}
It took four decades since the invention of forcing for set theorists and to ask (and answer) what post factum seems as one of the most natural questions regarding forcing. Is the ground model a definable class of its set-forcing extensions? Laver published the positive answer in a paper mainly concerned with whether rank-into-rank cardinals can be created by small forcing \cite{laver:groundmodel}. Woodin obtained the same result independently, and it appeared in the appendix of \cite{woodin:groundmodel}.
\begin{theorem}[Laver, Woodin]\label{th:zfcgroundmodeldefinable}
Suppose $V$ is a model of $\ZFC$, $\p\in V$ is a forcing notion, and $G\subseteq\p$ is $V$-generic. Then in $V[G]$, the ground model $V$ is definable from the parameter $P(\gamma)^V$, where $\gamma=|\p|^V$.
\end{theorem}
Indeed, it follows from the proof of Theorem~\ref{th:zfcgroundmodeldefinable}
that this definition of the ground model is uniform across all its set-forcing extensions. There is a first-order formula which, using the ground model parameter $P(\gamma)^V$ where $\gamma=|\p|^V$,\footnote{The parameter $P(\gamma)^V$ for $\gamma=|\p|^V$ appeared in Woodin's statement of Theorem~\ref{th:zfcgroundmodeldefinable}, while Laver's statement of it used the less optimal parameter $V_{\delta+1}$ for $\delta=\gamma^+$.} defines the ground model in any set-forcing extension\footnote{It is known that the ground model may not be definable in a class-forcing extension satisfying~$\ZFC$. A counterexample to definability, attributed to Sy-David Friedman, is the forcing extension by the class Easton product adding a Cohen subset to every regular cardinal \cite{hamkins:undefinableclassgroundmodel}.}. Before Theorem~\ref{th:zfcgroundmodeldefinable}, properties of the forcing extension in relation to the ground model could be expressed in the forcing language using the predicate $\check V$ for the ground model sets. But having a uniform definition of ground models in their set-forcing extensions was an immensely more powerful result that opened up rich new avenues of research. Hamkins and Reitz used it to introduce the \emph{Ground Axiom}, a first-order assertion  that a universe is not a nontrivial set-forcing extension \cite{reitz:groundaxiom}. Research on the Ground Axiom in turn grew into the set-theoretic geology project that reverses the forcing construction by studying what remains from a model of set theory once the layers created by forcing are removed \cite{FuchsHamkinsReitz:Set-theoreticGeology}.  Woodin made use of Theorem~\ref{th:zfcgroundmodeldefinable} in studying generic multiverses---collections of set-theoretic universes that are generated from a given universe by closing under generic extensions and ground models \cite{woodin:groundmodel}. In addition, Theorem~\ref{th:zfcgroundmodeldefinable} proved crucial to Woodin's pioneering work on suitable extender models, a potential approach to constructing the canonical inner model for a supercompact cardinal \cite{woodin:suitableextendermodels}.

In this article we investigate ground model definability for models of fragments of $\ZFC$, particularly of $\ZF+\DC_\delta$ and of $\ZFC^-$, and we obtain both positive and negative results.

Laver's proof~\cite{laver:groundmodel} that ground models of $\ZFC$ are definable in their set-forcing extensions uses Hamkins' techniques and results on pairs of models with the $\delta$-cover and $\delta$-approximation properties.
\begin{definition}[Hamkins \cite{hamkins:coverandapproximations}] \label{def:coverandextprop}
Suppose $V\subseteq W$ are transitive models of (some fragment of) $\ZFC$ and $\delta$ is a cardinal in $W$.
\begin{itemize}
\item[(1)] The pair $V\subseteq W$ satisfies the $\delta$-\emph{cover property} if for every $A\in W$ with $A\subseteq V$ and $|A|^W<\delta$, there is $B\in V$ with $A\subseteq B$ and $|B|^V<\delta$.
\item[(2)] The pair $V\subseteq W$ satisfies the $\delta$-\emph{approximation property} if  whenever $A\in W$ with $A\subseteq V$ and $A\cap a\in V$ for every $a$ of size less than $\delta$ in $V$, then $A\in V$.
\end{itemize}
\end{definition}

%Slightly weaker forms of $\delta$-cover and $\delta$-approximation properties for models of $\ZFC$ are presented in~\cite{woodin} and credited to Hamkins also\footnote{The two slightly different characterizations of the $\delta$-cover property of~\cite{hamkins:coverandapproximations} versus~\cite{wooden} are clearly equivalent in the $\ZFC$ context, but by itself the $\delta$-approximation property of~\cite{woodin} appears to be strictly weaker than the one of Definition~\ref{def:coverandextprop} (2).}. However, if two models of $\ZFC$ satisfy both the $\delta$-cover and the $\delta$-approximation property as defined in~\cite{woodin}, then it can be shown that they also satisfy the two properties as defined in Definition~\ref{def:coverandextprop}.
\noindent Pairs of the form the ground model with its forcing extension, $V\subseteq V[G]$, satisfy the $\delta$-cover and $\delta$-approximation properties for any cardinal $\delta\geq\gamma^+$, where $\gamma$ is the size of the forcing poset. This fact is proved in \cite{laver:groundmodel}, and it is an immediate corollary of  Lemma 13 of \cite{hamkins:coverandapproximations}, which easily generalizes to the following theorem.
\begin{theorem}[Hamkins]\label{th:gapcoverapprox}
Suppose  $\delta$ is a cardinal and $\p$ is a poset which factors as $\R*\dot \q$, where $\R$ is nontrivial\footnote{Here, and elsewhere in this article, a poset is \emph{nontrivial} if it necessarily adds a new set.}  of size less than $\delta$ and $\forces_\R \dot \q\text{ is strategically }\lt\delta\text{-closed}$. Then the pair $V\subseteq V[G]$ satisfies the $\delta$-cover and $\delta$-approximation properties for any forcing extension $V[G]$ by $\p$.
\end{theorem}
\begin{theorem}[Hamkins, see \cite{laver:groundmodel}]\label{th:coverandapprox}
Suppose $V$, $V'$ and $W$ are transitive  models of $\ZFC$, $\delta$ is a regular cardinal in $W$, the pairs $V\subseteq W$ and $V'\subseteq W$ have the $\delta$-cover and $\delta$-approximation properties, $P(\delta)^V=P(\delta)^{V'}$, and $(\delta^+)^V=(\delta^+)^W$. Then $V=V'$.
\end{theorem}
Laver's proof of Theorem~\ref{th:zfcgroundmodeldefinable} proceeds by combining his weak version of Theorem~\ref{th:gapcoverapprox} with Hamkins' uniqueness Theorem \ref{th:coverandapprox} as follows. A forcing extension $V[G]$ by a poset $\p$ of size $\gamma$ has the $\delta$-cover and $\delta$-approximation properties for $\delta=\gamma^+$, and moreover it holds that $(\delta^+)^V=(\delta^+)^{V[G]}$. It is not difficult to see that there is an unbounded definable class $C$ of ordinals such that for every $\lambda\in C$, the $\delta$-cover and $\delta$-approximation properties reflect down to the pair $V_\lambda\subseteq V[G]_\lambda$ and both $V_\lambda$ and $V[G]_\lambda$ satisfy a large enough fragment of $\ZFC$, call it $\ZFC^*$, for the proof of Theorem~\ref{th:coverandapprox} to go through. Letting $s=P(\delta)^V$, the sets $V_\lambda$, for $\lambda\in C$, are then defined in $V[G]$ as the unique transitive models $M\models \ZFC^*$ of height $\lambda$, having $P(\delta)^M=s$ such that  the pair $M\subseteq V[G]_\lambda$ has the $\delta$-cover and $\delta$-approximation properties. Finally,  we can replace the parameter $s=P(\delta)^V$ with $P(\gamma)^V$ by observing that $P(\gamma)^V$ is definable from $P(\delta)^V$ in $V[G]$ using the delta-approximation property (see Section~\ref{sec:zfcminus} in the paragraph before Theorem 4.1 for the argument).

Forcing constructions over models of $\ZF$ can be carried out in some overarching $\ZFC$ context because the essential properties of forcing such as the definability of the forcing relation and the Truth Lemma do not require choice. Also, forcing over models of $\ZF$ preserves $\ZF$ to the forcing extension.\footnote{All these facts follow by examining Shoenfield's proofs of them in \cite{shoenfield:forcing} for the $\ZFC$ context. Since maximal antichains need not exist without choice, generic filters must meet all dense subsets.} Is every model of $\ZF$ definable in its set-forcing extensions? Although at the outset, it might appear that the  $\delta$-cover and $\delta$-approximation properties machinery, used to prove the definability of $\ZFC$-ground models, isn't applicable to models without full choice, we will show that much of it can be salvaged with only a small fragment of choice. In Section~\S\ref{sec:zf}, we prove an analogue of Theorem~\ref{th:coverandapprox}  for models of $\ZF+\DC_\delta$ (Theorem~\ref{th:coverandapproxzf}) and derive from it a partial definability result for ground models of $\ZF+\DC_\delta$ and forcing extensions by posets admitting a gap at $\delta$. Posets admitting a gap at $\delta$ are particularly suited to forcing over models of $\ZF+\DC_\delta$ because they also preserve $\DC_\delta$ to the forcing extension (Theorem~\ref{th:dcpreservation}).

\Levy~\cite{levy:DCkappa} introduced the dependent choice axiom variant $\DC_\delta$, for an ordinal $\delta$, asserting that for any nonempty set $S$ and any binary relation $R$, if for each sequence $s \in S^{\lt\delta}$ there is a $y\in S$ such that $s$ is $R$-related to $y$, then there is a function $f:\delta\to S$ such that $f\restrict\alpha \;R f(\alpha)$ for each $\alpha<\delta$. It is easy to see that $\DC_\delta$ implies the choice principle $\AC_\delta$, the assertion that indexed families $\{A_\xi\mid \xi<\delta\}$ of nonempty sets have choice functions. The full $\AC$  is clearly equivalent to the assertion $\forall\delta \;\DC_\delta$, while $\AC_\delta$ is much weaker than $\DC_\delta$, as it provides choice functions only for already well-ordered families of nonempty sets.\footnote{Indeed, for any fixed $\delta$, the principle $\AC_\delta$ does not imply $\DC_\omega$, while the assertion $\forall\delta \AC_\delta$ does imply $\DC_\omega$ but not $\DC_{\omega_1}$ (see Chapter~8 in ~\cite{Jech1973:AxiomChoice}).} Some of the natural models of $\ZF+\DC_\delta$ arise as symmetric inner models of forcing extensions and models of the form $L(V_{\delta+1})$. In \cite{hamkins:gapforcing}, Hamkins defined that a poset $\p$ \emph{admits a gap} at a cardinal $\delta$ if it factors as $\R*\dot \q$, where $\R$ is nontrivial forcing of size less than $\delta$, and it is forced by $\R$ that $\dot \q$ is strategically $\lesseq\delta$-closed. By Theorem~\ref{th:gapcoverapprox}, a $\ZFC$ ground model with a forcing extension by a poset admitting a gap at a cardinal $\delta$ satisfy the $\delta$-cover and $\delta$-approximation properties. Indeed the analogous result for $\ZF+\DC_\delta$ holds as well (Theorem~\ref{th:forcingcoverandapproxzf}).
\begin{main1}\label{main1}
Suppose $V$ is a model of $\ZF+\DC_\delta$, $\p\in V$ is a forcing notion admitting a gap at $\delta$, and $G\subseteq\p$ is $V$-generic. Then in $V[G]$, the ground model $V$ is definable from the parameter $P(\delta)^V$.
\end{main1}

Models of the theory $\ZFC^-$, known as set theory without powerset, are used widely throughout set theory. Typically, but  not necessarily, these have a largest cardinal $\kappa$. The canonical ones are models $\her{\kappa}$, which are collections of all sets of hereditary size at most $\kappa$ for some cardinal $\kappa$. Models of $\ZFC^-$ also play a prominent role in the theory of smaller large cardinals, many of which, such as weakly compact, remarkable, unfoldable, and Ramsey cardinals, are characterized by the existence of elementary embeddings of $\ZFC^-$ models. While set theorists often think of $\ZFC^-$ as simply the axioms of $\ZFC$ with the powerset axiom removed, the situation is more complex. Indeed, removing the powerset axiom from $\ZFC$ has many surprising consequences as illustrated by the work of Zarach. Zarach showed that without the powerset axiom, the collection and replacement schemes are not equivalent, and that the axiom of choice does not imply that every set can be well-ordered \cite{Zarach1996:ReplacmentDoesNotImplyCollection,zarach:unions_of_zfminus_models}. Together with Hamkins, we continued Zarach's project in our article \cite{zfcminus:gitmanhamkinsjohnstone}, whose theme was the importance of including collection and not just replacement in what we understand to be set theory without powerset. We showed that a number of crucial set-theoretic results, such as the \Los\ Theorem for ultrapowers, or Gaifman's theorem that a $\Sigma_1$-elementary cofinal embedding is fully elementary, may fail for models of replacement but not collection in the absence of powerset. In light of these facts, we define $\ZFC^-$ as in~\cite{zfcminus:gitmanhamkinsjohnstone} to mean the theory $\ZFC$ without the powerset axiom, with the replacement scheme replaced by the collection scheme and with the axiom of choice replaced by the assertion that every set can be well-ordered.

Forcing over models of $\ZFC^-$ preserves $\ZFC^-$ to the forcing extension and the rest of standard forcing machinery carries over as well.\footnote{It is an open question whether forcing extensions of models of set theory without powerset where replacement is used in place of collection continue to satisfy replacement.} However, in Section~\S\ref{sec:zfcminus}, we show that Laver's ground model definability result cannot be generalized to $\ZFC^-$ ground models. Using forcing, we produce a $\ZFC$ universe with a cardinal $\kappa$ such that ground model definability fails for $\her{\kappa}$. In this case, ground model definability is violated in the strongest possible sense because $\her{\kappa}$ has a set-forcing extension in which it is not definable even using a parameter from the extension. We can set up the preparatory forcing so that $\kappa$ is any ground model cardinal and so that the forcing extension violating ground model definability is by a poset of the form $\Add(\delta,1)$\footnote{We denote by $\Add(\delta,\gamma)$, where $\delta$ is an infinite cardinal and $\gamma$ is any cardinal, the poset which adds $\gamma$-many Cohen subsets to $\delta$, using conditions of size less than $\delta$.} for some regular cardinal $\delta<\!<\kappa$. It will follow from our arguments that there is always a countable transitive model of $\ZFC^-$ violating ground model definability.
\begin{main2}\label{th:zfcminusundefinable}
Assume that $\delta, \kappa$ are cardinals such that $\delta$ is regular and either $2^{\lt\delta}\!<\!\kappa$\, or \,$\delta\!=\!\kappa\!=\!2^{\lt\kappa}$ holds. If $V[G]$ is a forcing extension by $\Add(\delta,\kappa^+)$,
% \emph{(}for regular $\delta$ with $2^{\lt\delta}<\kappa$ or $\delta=\kappa$ inaccessible\emph{)}
then $\her{\kappa}^{V[G]}$ is not definable in its forcing extension by $\Add(\delta,1)$. It follows that there is a countable transitive model of $\ZFC^-$ that is not definable in its Cohen forcing extension.
\end{main2}
\noindent For instance, it follows that it is consistent for $H_{\omega_2}$ to fail to be definable in its Cohen forcing extension $H_{\omega_2}[g]$.

The proof method of Main~Theorem~2 has an interesting consequence for ground model definability of $\ZFC$ models. We noted earlier that the natural parameter $P(|\p|^+)$ can be improved to $P(|\p|)^V$, but we will show that it cannot be improved any further.
\begin{theorem}\label{th:parameterlowerbound}
It is consistent that a ground model $V$ cannot be defined in its Cohen extension $V[g]$ using any parameter of hereditary size less than $2^\omega$.
\end{theorem}

\section{ZF-Forcing preliminaries}\label{sec:preliminaries}
Many of the concepts from the standard forcing toolbox use the axiom of choice in both obvious and subtle ways. For instance, nice names, which play a crucial role in forcing constructions, need not exist in choiceless models. Also, without choice, the concept of closure of a forcing notion loses much of its potency because a nontrivial infinite poset may, for instance, be vacuously countably closed simply because there are no infinite descending chains. Indeed, it is not difficult to see that the assertion that $\lesseq\delta$-closed forcing does not add new $\delta$-sequences of ground model sets is actually equivalent to $\DC_\delta$. Another issue which arises when forcing over models satisfying only a fragment of choice is that this fragment need not be preserved to the forcing extension. For instance, Monro showed that it is possible to have a model of $\ZF+\DC_\delta$ for a cardinal $\delta>\!>\omega$ that has a  forcing extension that does not even satisfy $\AC_\omega$ \cite{monro:preservingchoicezf}. In this section, we will briefly discuss how to adapt certain forcing related concepts, such as nice names and full names, to the choiceless setting. We will also show that posets admitting a gap at $\delta$ preserve $\DC_\delta$ to the forcing extension. This is critical for our results because the generalized uniqueness theorem (Theorem~\ref{th:coverandapproxzf}), which we prove in Section~\ref{sec:zf}, requires all three models to satisfy $\DC_\delta$, and so it can only apply to models $V,V'\subseteq V[G]$ provided that $V[G]\models\DC_\delta$.

Suppose $\p$ is a poset in a model $V\models \ZF$ and $\sigma$ is a $\p$-name. As a natural replacement for nice names, we define that a \emph{good name} for a subset of $\sigma$ is any $\p$-name $\tau$ such that $\tau\of\dom(\sigma)\times\p$. It is easy to see that good names share the defining property of nice names, namely that if $\sigma,\mu$ are $\p$-names, then there is a good $\p$-name $\tau$ for a subset of $\sigma$ such that $\one\forces \left(\mu\subseteq\sigma\to \mu=\tau\right)$. Therefore, good names can be used instead of nice names in the construction of the canonical names for the $V_\alpha$-hierarchy of the forcing extension. We define these by recursion as follows: $\sigma_0=\emptyset$, $\sigma_{\alpha+1}=\{\la\tau,\one\ra\mid \tau \text{ is a good name for a subset of }\sigma_\alpha\}$, and $\sigma_\lambda=\bigcup_{\alpha<\lambda} \sigma_\alpha$ for limit ordinals $\lambda$. It then follows that $(\sigma_\alpha)_G=V[G]_\alpha$ for every $V$-generic filter $G$. Moreover, assuming we used a flat pairing function for constructing $\p$-names,\footnote{A \emph{flat pairing function} is a way of defining ordered pairs which ensures that if $a,b\in V_\alpha$, then so does the ordered pair of $a$ and $b$.} we get that $\sigma_\alpha\subseteq V_\alpha$ for every $\alpha\geq \gamma\cdot\omega$, where $\gamma$ is the rank of $\p$, and hence $V_\alpha[G]=V[G]_\alpha$ for all sufficiently large $\alpha$. Having the canonical names $\sigma_\alpha$ and knowing that good names suffice to represent all subsets of $\sigma_\alpha$, we get that for any $\p$-name $\sigma$, there is an ordinal $\gamma$ such that whenever $p\in\p$ is a condition and $\mu$ is a $\p$-name such that $p\forces \mu\in\sigma$, then there is another $\p$-name $\tau\in V_\gamma$ such that $p\forces \mu=\tau$.

There are a few different approaches to defining a two-step forcing iteration $\p*\dot \q$ in $\ZFC$, all of which can be shown to have the desired properties in the absence of choice as well. For concreteness, we use full names. For a poset $\p$, a $\p$-name $\tau$ is called \emph{full} if $\tau=\text{dom}(\tau)\times\{\one\}$ and whenever $p\in\p$ and $\sigma$ is a $\p$-name such that $p\forces\sigma\in\tau$, then there is a $\sigma'\in\text{dom}(\tau)$ such that $p\forces\sigma=\sigma'$. In models of $\ZFC$, restricting to full names comes without a loss, since for every $\p$-name $\tau$ such that $\one\forces \tau\neq \emptyset$, there is a full name $\tau'$ such that $\one\forces \tau=\tau'$. The argument to see this uses nice names, the canonical names $\sigma_\alpha$, and the technique of \emph{mixing} to verify that whenever $p\forces \sigma\in\tau$, then there is a another name $\sigma'$ such that $p\forces \sigma=\sigma'$ and $\one \forces \sigma'\in\tau$. In models of $\ZF$, good names again take on the role of nice names in this argument and,  provided that for some fixed name $\tau_0$, we have that $\one\forces\tau_0\in \tau$, we can mix $\tau_0$ and $\sigma$ to create the required name $\sigma'$ without any need for maximal antichains. Thus, in $\ZF$, whenever a $\p$-name $\tau$ has an element $\tau_0$ such that $\one\forces \tau_0\in\tau$, it follows that a full name for $\tau$ exists. Since we can insist that any $\p$-name that is forced by $\one$ to be a poset comes with the top element $\dot \one$, it follows that, in models of $\ZF$, every $\p$-name for a poset is forced by $\one$ to be equal to a full name. Thus, using full names, we define that $\p*\dot \q$, where $\p$ is a poset and $\dot \q$ is a full name for a poset, is the partial order consisting of conditions $(p,\dot q)$, where $p\in\p$ and $q\in\text{dom}(\dot \q)$, ordered so that $(p,\dot q)\leq (p',\dot q')$ whenever $p\leq p'$ and $p\forces \dot q\leq \dot q'$.\footnote{An alternative approach to defining a two-step iteration that does not require $\dot \q$ to be a full name, starts with a proper class of conditions $(p,\dot q)$, where $p\in\p$ and $\one\forces \dot q\in\dot\q$, that is later cut  down to a set using good names together with the names $\sigma_\alpha$. Although this approach appears to skirt the need for a top element, most arguments involving iterations rely on the ability to turn a name $\dot q$ forced by some condition $p$ to be an element of $\dot \q$ into a name $\dot q'$ such that $p\forces \dot q=\dot q'$ and $(p,\dot q')\in \p*\dot\q$, which needs mixing or alternatively a top element.}

With the technical preliminaries out of the way, we now proceed to argue that a general class of posets, extending those admitting a gap at $\delta$, preserve $\DC_\delta$ to the forcing extension.
\begin{theorem} \label{th:DCPreservation1}Suppose that $V\models \ZF+\DC_\delta$ for an ordinal $\delta$ and $\p$ is well-orderable of order type at most $\delta$. Then every  forcing extension $V[G]$ by $\p$ is a model of $\DC_\delta$.\footnote{In the case when $\p$ collapses the cardinality of $\delta$, what really is preserved is $\DC_\gamma$ for $\gamma=|\delta|^{V[G]}$.}
\end{theorem}

\begin{proof}
In $V[G]$, suppose $R$ is a relation and $A$ is a set such that for all $s\in A^{\lt\delta}$ there is a $y\in A$ with $sRy$. Fix $\p$-names $\dot{A}$ and $\dot R$ such that $\dot A_G=A$ and $\dot R_G=R$, and also fix a condition $p\in G$ forcing the hypothesis of $\DC_\delta$ for $\dot{A}$ and $\dot{R}$.  By the previous remarks, we can find an ordinal $\gamma$ such that whenever  $p\forces \sigma\in \Adot$ for some $\p$-name $\sigma$, then there is another $\p$-name $\mu\in V_\gamma$ such that $p\forces \sigma=\mu$. Let $B=\{\sigma\in V_\gamma \st p\forces\sigma\in\Adot\}$. Recall also that if $s=\<\sigma_\xi\st \xi<\alpha>\in V$ is any sequence of $\p$-names, then there is a canonical $\p$-name $\tau_s$ such that $\one\forces``\tau_s\text{ is an }\alpha\text{-sequence}"$ and $\one\forces \tau_s(\xi)=\sigma_\xi$ for all $\xi< \alpha$. We shall define a binary relation $R^*$ on $B^{\lt\delta}\times B$ as follows. For $s\in B^{\lt\delta}$ and $\sigma\in B$, we define that $sR^*\sigma$ whenever $p\forces \tau_s\dot{R}\sigma$. We now argue that the hypothesis of $\DC_\delta$ is satisfied for $B$ and $R^*$. Towards this end, we suppose that $s\in B^{\alpha}$ for some $\alpha<\delta$. By the $\DC_\delta$ hypothesis for $\Adot$ and $\dot{R}$ forced by $p$, it follows that $p\forces \exists x\in\Adot\; \tau_s\dot{R}x$. Since $\p$ is well-orderable, there exists a maximal antichain below $p$ of conditions $q$ such that for some name $\sigma$, we have $q\forces ``\sigma\in \dot A\text{ and } \tau_s \dot R \sigma"$. Now using $\AC_\delta$, for each such $q$ we choose a name $\sigma_q$ and mix to obtain a single $\p$-name $\sigma$ such that $p\forces ``\sigma\in\Adot\text{ and } \tau_s\dot{R}\sigma$". Without loss of generality, we may assume that $\sigma\in V_\gamma$, and so $sR^*\sigma$ as desired. Now applying $\DC_\delta$, there is a sequence $s=\<\sigma_\xi\st\xi<\delta>$ such that $s\restrict \alpha\;R^* \sigma_\alpha$ for each $\alpha<\delta$.  By the definition of $R^*$, it follows that $p\forces \tau_s\restrict \alpha\;\dot{R}\;\tau_s(\alpha)$ for all $\alpha<\delta$. The interpretation $(\tau_s)_G$ is thus the desired $\delta$-sequence witnessing $\DC_\delta$ for $A$ and $R$ in $V[G]$.
\end{proof}

Recall that a poset is \emph{strategically} $\lt\gamma$-\emph{closed} if in the game of ordinal length $\gamma$ in which two players alternatively select conditions from it to construct a descending $\gamma$-sequence with the second player playing at limit stages, the second player has a strategy that allows her to always continue playing; a poset is \emph{strategically} $\lesseq\gamma$-\emph{closed} if the corresponding game has length $\gamma+1$.
\begin{theorem}\label{th:DCPreservation2} Suppose that $V\models \ZF+\DC_\delta$ for an ordinal $\delta$ and $\p$ is a strategically $\lesseq\delta$-closed poset in $V$. Then every forcing extension $V[G]$ by $\p$ is a model of $\DC_\delta$.
\end{theorem}
\begin{proof}
We shall follow the proof of Theorem~\ref{th:DCPreservation1}, while avoiding the need for mixing (the only place where the well-orderability of $\p$ was used) by using the strategic closure property instead. To simplify the presentation, let us assume first that $\p$ is $\lesseq\delta$-closed and outline at the end of the proof how the argument can be modified in the case when $\p$ is merely strategically $\lesseq\delta$-closed. In $V[G]$, suppose $R$ is a relation  and $A$ is a set such that for all $s\in A^{\lt\delta}$ there is a $y\in A$ with $sRy$. Fix a $\p$-name $\dot{A}$ and $\dot R$ such that $\dot A_G=A$ and $\dot R_G=R$, and also fix a condition $p\in G$ forcing the hypothesis of $\DC_\delta$ for $\dot{A}$ and $\dot{R}$.  Next, we find an ordinal $\gamma$ such that whenever a condition $q \forces \sigma\in \Adot$ for some $\p$-name $\sigma$, then there is another $\p$-name $\mu\in V_\gamma$ such that $q\forces \sigma=\mu$. We shall argue that it is dense below $p$ to have conditions $q$ forcing the existence of a sequence witnessing $\DC_\delta$ for $\dot A$ and $\dot R$. Towards this end, we fix some $q\leq p$. Let $B=\{\la r, \sigma\ra\mid r\leq q, \sigma\in V_\gamma, r\forces\sigma\in\Adot\}$. We shall define a binary relation $R^*$ on $B^{\lt\delta}\times B$ as follows. Suppose $z=\<\<r_\xi,\sigma_\xi>:\xi<\alpha>$ is a sequence of elements of $B$ for some $\alpha<\delta$ and let $s=\<\sigma_\xi :\xi<\alpha>$. If $\<r_\xi:\xi<\alpha>$ is a descending sequence of conditions, then we define that $zR^*\la r,\sigma\ra$ whenever $r$ is below all the $r_\xi$ and $r\forces \tau_s \dot R\sigma$. Otherwise, we define that $zR^*\la r,\sigma\ra$ for every $\la r,\sigma\ra\in B$. We now argue that the hypothesis of $\DC_\delta$ is satisfied for $B$ and $R^*$. If $z=\<\<r_\xi,\sigma_\xi>:\xi<\alpha>$ with $\la r_\xi\mid\xi<\alpha\ra$ descending, then since $\p$ is (much more than) $\lesseq\alpha$-closed, there is a condition $r^*\in\p$ below all the $r_\xi$. It is clear that $r^*\forces \sigma_\xi\in\dot A$ for all $\xi<\alpha$ and thus, $r^*\forces \exists x\in\Adot\; \tau_s\dot{R}x$. In contrast to the proof of Theorem~\ref{th:DCPreservation1}, we cannot use mixing to obtain a witnessing name, but instead we strengthen $r^*$ to a condition $r$ for which there exists a $\p$-name $\sigma$ such that $r\forces ``\sigma\in\Adot\text{ and } \tau_s\dot{R}\sigma$". Without loss of generality, we may assume that $\sigma\in V_\gamma$, and so $z\;R^*\<r,\sigma>$, as desired.

Now applying $\DC_\delta$ in $V$, there is a sequence $z=\<\<r_\xi,\sigma_\xi>:\xi<\delta>$ of elements of $B$  such that $z\restrict \xi\;R^*\;\<r_\xi,\sigma_\xi>$ for each $\xi<\delta$. We let $s=\la \sigma_\xi\mid\xi<\delta\ra$ and consider $\tau_s$. By induction on $\xi$, it is easy to see that $\la r_\xi\mid \xi<\delta\ra$ is a descending sequence of conditions in $\p$ and each $r_\xi\forces \tau_s\restrict \xi\dot R\tau_s(\xi)$.  Since $\p$ is $\lesseq\delta$-closed, there is a condition $r\in\p$ below all the $r_\xi$. By the definition of $R^*$, it follows that $r\forces \tau_s\restrict \xi\;\dot{R}\;\tau_s(\xi)$ for all $\xi<\delta$. Thus, $r\leq q$ forces that there is a sequence witnessing $\DC_\delta$ for $\dot A$ and $\dot R$. This proves the theorem in the case when $\p$ is $\lesseq\delta$-closed.

It is straightforward to modify the argument for the case when $\p$ is merely strategically $\lesseq\delta$-closed, say with winning strategy $\Sigma$ for player II. The definition of $R^*$ has to be modified to insist that the descending sequence $\<r_\xi:\xi<\alpha>$ is built according to $\Sigma$, and when arguing that $R^*$ satisfies the hypothesis for $\DC_\delta$, the condition $r$ below all the $r_\xi$ must be chosen according to the strategy $\Sigma$.
\end{proof}
Theorem \ref{th:DCPreservation1} and \ref{th:DCPreservation2} have the immediate corollary:
\begin{theorem}\label{th:dcpreservation}
Suppose that $V\models \ZF+\DC_\delta$ for a cardinal $\delta$ and $\p\in V$ is a poset which factors as $\R*\dot \q$, where $|\R|\leq\delta$ and $\forces_{\R}``\dot \q\text{ is strategically }\lesseq\delta$-closed. Then every forcing extension $V[G]$ by $\p$ is a model of $\DC_\delta$. In particular, posets admitting a gap at $\delta$ preserve $\DC_\delta$ to the forcing extension.
\end{theorem}
\noindent Posets described in the hypothesis of Theorem~\ref{th:dcpreservation}, with the additional assumption that $\R$ is nontrivial, are said to admit a \emph{closure point} at $\delta$. These forcing notions were introduced by Hamkins in \cite{hamkins:coverandapproximations} as a significant generalization of posets admitting a gap at $\delta$. As we already noted, it follows from Theorem~\ref{th:gapcoverapprox} that the ground model and its forcing extension by a poset admitting a gap at $\delta$ have the $\delta$-cover and $\delta$-approximation properties, but if the forcing extension is by a poset admitting a closure point at $\delta$, then we are only guaranteed to have the $\delta^+$-cover and $\delta^+$-approximation properties. Chiefly because of this difference, we succeed in  showing that every ground model of $\ZF+\DC_\delta$ is definable in its set-forcing extensions by posets admitting a gap at $\delta$, while the analogous fact about closure point forcing is an open question (see Section~\ref{sec:questions}).

%Posets $\R*\dot\q$, as in the hypothesis of Theorem~\ref{th:dcpreservation}, where $\R$ is nontrivial forcing, are said to \emph{admit a closure point} at $\delta$. Such forcing notions were introduced by Hamkins, who showed, in the $\ZFC$ context, that posets admitting a closure point at $\delta$ necessarily have the $\delta^+$-cover and $\delta^+$-approximation properties \cite{hamkins:coverandapproximations}. Indeed, this crucial result does not need the full axiom of choice but also holds in our $\ZF+\DC_\delta$ context (Theorem~\ref{th:forcingcoverandapproxzf}). While we shall show in the next section that every ground model of $\ZF+\DC_\delta$ is definable in its set-forcing extensions by posets admitting a gap at $\delta$, the analogous fact about closure point forcing is an open question (see Section~\ref{sec:zfcminus}).
\section{Definable $\rm ZF$-ground models}\label{sec:zf}
In this section, we show that models of $\ZF+\DC_\delta$ are uniformly definable in their set-forcing extensions  by  posets admitting a gap at $\delta$. Recall that Laver's proof of $\ZFC$ ground model definability combined Hamkins' uniqueness theorem (Theorem~\ref{th:coverandapprox}) with the fact that the ground model and its forcing extension by $\p$ always have the $\delta$-cover and $\delta$-approximation property for $\delta>|\p|$.  Following this strategy, we proceed by first extending the uniqueness theorem to pairs of models of (a fragment of) $\ZF+\DC_\delta$ with the $\delta$-cover and $\delta$-approximation properties. Continuing to follow Laver, this would yield only that a ground model of $\ZF+\DC_\delta$ is definable its forcing extensions by well-ordered posets of size less than $\delta$. But with the help of the following $\ZF+\DC_\delta$ analogue of Theorem~\ref{th:gapcoverapprox}, we are able to significantly expand the class of posets for which $\ZF+\DC_\delta$ ground models are definable, to non-well-orderable posets also.
\begin{theorem}\label{th:forcingcoverandapproxzf}
Suppose $V\models \ZF+\DC_\delta$ for a cardinal $\delta$ and $\p$ is a poset which factors as $\R*\dot \q$, where $\R$ is nontrivial of size less than $\delta$ and $\forces_\R \dot \q$ is strategically $\lt\delta\text{-closed}$. Then the pair $V\subseteq V[G]$ satisfies the $\delta$-cover and $\delta$-approximation properties for any forcing extension $V[G]$ by $\R*\dot\q$. Indeed, if $\delta=\gamma^+$, then $\DC_\gamma$ suffices.
\end{theorem}
\begin{proof}
The proof is essentially identical to that of Lemma 12 in \cite{hamkinsjohnstone:unfoldability}, except that one needs to exercise care wherever the full axiom of choice is used. For instance, the $\delta$-cover property is verified by observing that it holds for each step of the forcing. For the second step, we use that $\lt\delta$-closed forcing does not add new $\lt\delta$-sequences of ground model elements, which in our case relies on the preservation of $\DC_\delta$ by the first step of the forcing (Theorem~\ref{th:DCPreservation1}). Another key step of that proof uses mixing in $\R$, which we are able to do as well because $\R$ is well-ordered of size less than $\delta$ and $\AC_\delta$ holds in $V$.
\end{proof}

\noindent Thus, in particular, the pair consisting of a model $V\models\ZF+\DC_\delta$ and its forcing extension $V[G]$ by a poset admitting a gap at $\delta$ has the $\delta$-cover and $\delta$-approximation properties. Crucially, it also follows that such a pair has the $\delta^+$-cover and $\delta^+$-approximation properties. The additional cover and approximation properties will allow us to fulfill a hypothesis of the generalized uniqueness theorem that we are about to state and prove.

As in the proof of $\ZFC$ ground model definability (sketched in the introduction), we will eventually need that the uniqueness theorem holds for some fragment of $\ZF+\DC_\delta$ such that there is a proper definable class of ordinals $\alpha$ for which $V_\alpha$ is a model of this fragment. We will denote by $\Z$ the fragment of $\ZF$ consisting of Zermelo set theory (without choice) together with the axiom asserting that the universe is the union of the von Neumann hierarchy.\footnote{Interestingly, over Zermelo set theory, the axiom asserting that for every ordinal $\alpha$, $V_\alpha$ exists does not imply that the universe is the union of the von Neumann hierarchy. A counterexample model was constructed by Sam Roberts in response to a MathOverflow question \cite{roberts:vN}. Recall that the \emph{Zermelo} ordinals are defined by $Z(\emptyset)=\emptyset$, $Z(\alpha+1)=\{Z(\alpha)\}$, and $Z(\lambda)=\{Z(\alpha)\mid\alpha<\lambda\}$. Now consider the model $M$ obtained by starting with $V_{\omega+\omega}$, adding $Z(\omega+\omega)$ and closing under pairing, union, subsets, and powersets in $\omega$-many steps. It is not difficult to see that $M$ is a model of Zermelo set theory of height $\omega+\omega$. Thus, $M$ satisfies that $V_\alpha$ exists for every ordinal $\alpha$, but it is not a model of $\Z$.} Observe that if $V\models \ZF+\DC_\delta$, then for every limit ordinal $\lambda$, $V_\lambda\models \Z$, and if moreover $\lambda>\delta$, then $V_\lambda\models\DC_\delta$ as well. Previous versions of the uniqueness theorem have used different fragments of $\ZFC$. In his proof of ground model definability, Woodin argued that the uniqueness theorem holds for models of the theory ${\rm ZC}^{(\text{VN})}+\Sigma_1\text{-replacement}$, where ${\rm ZC}^{(\text{VN})}$ consists of Zermelo set theory with choice and the additional axiom that $V_\alpha$ exists for every ordinal $\alpha$, and $\Sigma_1\text{-replacement}$ is the replacement axiom for $\Sigma_1$-definable functions \cite{woodin:groundmodel}. Reitz, in his ground axiom paper \cite{reitz:groundaxiom}, argued that the uniqueness theorem holds for models of the theory $\ZFC_\delta$ (for a regular cardinal $\delta$), consisting of Zermelo set theory with choice, replacement for definable functions with domain $\lesseq\delta$, and the additional axiom asserting that every set is coded by a set of ordinals\footnote{In~\cite{reitz:groundaxiom}, this coding axiom is given formally as $\forall A\exists\alpha\in\ORD\exists E\subseteq\alpha\times\alpha\, \la \alpha,E\ra\cong \la \text{tc}(\{A\}),\in\ra$, but this formalization appears to be slightly too weak. In order to truly code every set by a set of ordinals, Reitz's arguments use the ability to \emph{decode} sets of ordinals into transitive sets, but without the replacement axiom, this may not be possible. It thus appears that Reitz's $\ZFC_\delta$ should include an additional requirement, namely that for all ordinals $\alpha$ and every well-founded extensional relation $E \subseteq \alpha \times \alpha$, the Mostowski collapsing map of $\la \alpha,E\ra$ exists.} below, replacement for functions with domain $\lesseq\delta$ turns out to superfluous because the ranges of the functions in question are contained in a set and therefore $\DC_\delta$ (which follows from choice) suffices to argue that their ranges are themselves sets.\footnote{Reitz's theory would still need replacement for functions on $\omega$ because the proof uses the existence of transitive closures, which does not follow from Zermelo set theory.} Over Zermelo set theory, the $\Sigma_1$-replacement axiom implies Reitz's coding axiom and strengthens the assertion that $V_\alpha $ exists for every $\alpha$ to our assertion that the universe is the union of the von Neumann hierarchy. Because our proof combines a weak version of coding which already follows from $Z^*+\DC_\delta$ with an induction that relies on the fact that the universe is the union of the von Neumann hierarchy, we avoid the need for any additional replacement or coding axioms.
\begin{theorem}\label{th:coverandapproxzf}
Suppose that $V$, $V'$, and $W$ are transitive models of $\Z+\DC_\delta$, for some regular cardinal $\delta$ of $W$. Suppose that the pairs $V\of W$ and $V'\of W$ have the $\delta$-cover and $\delta$-approximation properties, $P(\delta)^V=P(\delta)^{V'}$, and $\her{\delta}^W\cap V=\her{\delta}^V$. Then $V=V'$.
\end{theorem}
\begin{proof} We follow the main ideas of the proof for the $\rm ZFC$ context as presented in~\cite{laver:groundmodel} (Theorem 1, Lemmas 1.1 and 1.2) closely, but make a few significant changes to adapt the arguments to the $\rm ZF$ case. The proof of \cite{laver:groundmodel} proceeds, for instance, by arguing that $V$ and $V'$ have the same sets of ordinals, a condition which suffices to conclude that $V=V'$ only if both are models of $\ZFC$. Our changes to that argument are designed to overcome this and other uses of full choice.

First, we make a general observation about the Mostowski collapse in models of $\Z$ that will be used throughout the proof. Suppose that $E$ is a well-founded extensional relation on a set $A$. Even though replacement may fail to hold, because $E$ is a relation on a set, we can define the $E$-rank function $e:\ORD\to P(A)$. In particular, if the $E$-ranks are bounded, that is, there is an $\alpha$ such that $e(\alpha)=e(\alpha+1)$, then the Mostowski collapse maps into $V_\alpha$, and therefore exists as a set. Thus, in models of $\Z$, any  extensional well-founded set relation with bounded ranks is isomorphic to the $\in$-relation on a transitive set. For instance, it follows from this that $\her{\delta}^V=\her{\delta}^{V'}$. To see this, suppose that $A\in V$ is a transitive set having a bijection to some ordinal $\delta'\leq \delta$. This bijection imposes a relation $E$ on $\delta'$ corresponding to the $\in$-relation on $A$. Since $P(\delta)^{V}=P(\delta)^{V'}$, we have that $E\in V'$. Clearly, since $E$ codes $A$, the $E$-ranks are bounded and hence $V'$ can Mostowski collapse $E$ to recover $A$.

Note that, using the $\delta$-cover and $\delta$-approximation properties, it follows that $V$, $V'$, and $W$ all have the same ordinals. Next, we observe that a set $A\in V$ has size less than $\delta$ in $V$ if and only if it has size less than $\delta$ in $W$. Clearly, since $V\subseteq W$, then $|A|^W\leq |A|^V$. For the other direction, it suffices to observe that, since $V$ satisfies $\DC_\delta$, it follows that either a set there has size less than $\delta$ or there is an injection of $\delta$ into it. The same statement obviously holds for $V'$ and $W$ as well, and thus when dealing with sets of size less than $\delta$, it does not matter in which of the three models that size is computed. Using the hypothesis that $\her{\delta}^W\cap V=\her{\delta}^V$, we can extend this conclusion to transitive sets of size $\delta$. This additional assumption replaces an analogous hypothesis of Theorem~\ref{th:coverandapprox} that $(\delta^+)^V=(\delta^+)^W$, which appears insufficient without full choice, since a set in $V$ of size $\delta$ in $W$ may not even be well-orderable in $V$.

\emph{Claim 1:} Suppose that $T\in V\intersect V'$ is a transitive set and $A\in W$ is any subset of $T$ of size less than $\delta$. Then there exist a common cover $B\in V\intersect V'$ with $A\of B$ and a common bijection $f:B\to \delta'$ with $f\in V\intersect V'$ for some $\delta'\leq\delta$.

\emph{Proof of Claim 1:} Fix $A\in W$ with $A\of T$ and $|A|<\delta$. Using the $\delta$-cover property of the pair $V\subseteq W$, there is a cover $B_0\in V$ such that $A\subseteq B_0\subseteq T$ and $|B_0|<\delta$. Now we make the key observation that, since $T$ is transitive, we can, working in $V$ and using $\DC_\delta$, extend $B_0$ to an $\in$-extensional cover of size less than $\delta$. Thus, we may assume without loss of generality that $B_0$ is already extensional. The set $B_0$ may in turn be covered by an extensional $B_1\in V'$ with $B_1\subseteq T$ and $|B_1|<\delta$, this time using the $\delta$-cover property of the pair $V'\subseteq W$. Now we observe that $W$ can tell which subsets of $T$ are elements of $V$ or $V'$ by consulting $P(T)^V$ and $P(T)^{V'}$, both of which are elements of $W$. Thus, working in $W$ and using $\DC_\delta$ together with the regularity of $\delta$ to get through limit stages, we may obtain a sequence
\begin{displaymath}
A\subseteq B_0\subseteq B_1\subseteq \cdots\subseteq B_\xi\subseteq\cdots\text{ for }\xi<\delta
\end{displaymath}
with cofinally many $B_\xi\in V$ and cofinally many in $V'$ such that each $B_\xi\subseteq T$ is extensional and $|B_\xi|<\delta$. Let $B=\Union_{\xi<\delta}B_\xi$ and note that it has size at most $\delta$ in $W$ by $\DC_\delta$. Since $\delta$ is regular, it follows that if $a\in W$ is any set of size less than $\delta$, then $B\intersect a=B_\xi\intersect a$ for a sufficiently large $\xi$, and so $B\in V\cap V'$ by the $\delta$-approximation property. Thus, it remains to demonstrate the existence of the required bijection $f$. Since $\in$ is clearly extensional on $B$, we let $\pi:\<B,\in>\to \<b,\in>$ be the Mostowski collapse of $\<B,\in>$. Note that $\pi$ exists both in $V$ and in $V'$, by the uniqueness of the collapsing map, and so $b$ is a transitive set in $V\intersect V'$ with $|b|^W\leq\delta$. By our earlier remarks, it follows that $|b|^V\leq\delta$ as well, and  thus, we let $g\in V$ be a bijection from $b$ onto $\delta'$ for some $\delta'\leq\delta$. Since the bijection $g$ can be coded by a subset of $\delta'$, and $P^V(\delta)=P^{V'}(\delta)$ by assumption, it follows that $g\in V'$ also. Let $f=g\circ\pi$ be the composition map. Then $f\in V\intersect V'$, and $f:B\to \delta'$ is the desired a bijection that exists in $V\intersect V'$.~$\square$

\emph{Claim 2:} Suppose that $T\in V\intersect V'$ is a transitive set and $A\in W$ is any subset of $T$ of size less than $\delta$. Then $A\in V$ if and only if $A\in V'$.

\emph{Proof of Claim 2:} We assume that $A\in V$. By Claim~1, there is a cover $B\in V\intersect V'$ with $A\of B$ and a bijection $f\in V\intersect V'$ with $f:B\to \delta'$ for some $\delta'\leq\delta$. Since $V$ and $V'$ have the same subsets of $\delta$ by assumption and $f\image A\of \delta'\of \delta$, it follows that $f\image A \in V'$ and hence $A\in V'$.~$\square$

Remarkably, the $\delta$-approximation property now allows us to strengthen Claim 2 to apply to all $A\of T$, whether well-orderable or not.

\emph{Claim 3.} Suppose that $T\in V\intersect V'$ is a transitive set and $A\in W$ is any subset of $T$. Then  $A\in V$ if and only if $A\in V'$.

\emph{Proof of Claim 3.}  We assume that $A\in V$. Since $T\in V\intersect V'$, it follows that $A\of V'$, and so we may apply the $\delta$-approximation property of the pair $V'\of W$ to argue that $A\in V'$. Toward this end, we fix some $a\subseteq T$ in $V'$ of size less than $\delta$, and proceed to show that $a\intersect A\in V'$.  By Claim~2, it follows that $a$ is an element of $V$. Thus, $a\cap A$ is an element of $V$ of size less than $\delta$. Using Claim~2 once again, we have that $a\intersect A$ is an element of $V'$, which completes the argument that $A\in V'$.~$\square$

Finally, to prove the theorem, we first argue by induction that $V_\alpha^V=V_\alpha^{V'}$ for all ordinals $\alpha$. The limit step is trivial, and for the successor step, we assume inductively that $V_\xi^V=V_\xi^{V'}$ and apply Claim~3 to conclude that $V_{\xi+1}^V=P(V_\xi^V)=P(V_\xi^{V'})=V_{\xi+1}^{V'}$. It follows that $V=V'$ since the $\Z$-axioms include the assertion that the universe the union of the von Neumann hierarchy.
\end{proof}
As we already noted in the proof of Theorem~\ref{th:coverandapproxzf}, its hypothesis that $\her{\delta}^W\cap V=\her{\delta}^V$ replaces the analogous hypothesis of Theorem~\ref{th:coverandapprox} that $(\delta^+)^V=(\delta^+)^W$, which might be weaker in the choiceless context. However, in the presence of slightly more choice, namely $\DC_{\delta^+}$ rather than $\DC_\delta$, we claim that $(\delta^+)^V=(\delta^+)^W$ implies $\her{\delta}^W\cap V=\her{\delta}^V$. If $V\models \Z+\DC_{\delta^+}$ and $A\in V$, then either $A$ has size $\lesseq\delta$ in $V$, or there is an injection from $(\delta^+)^V$ into $A$. Thus, if $V\subseteq W$  with $(\delta^+)^V=(\delta^+)^W$ and  $A\in\her{\delta}^W\cap V$, then such an injection cannot exist and so $A$ has hereditary size at most $\delta$ in $V$, as desired.

The proof of Main~Theorem~1 now follows by combining Theorem~\ref{th:coverandapproxzf} together with Theorem~\ref{th:dcpreservation} and the proof of Theorem~\ref{th:zfcgroundmodeldefinable} from \cite{laver:groundmodel} sketched in the introduction.
\begin{main1}
Suppose $V$ is a model of $\ZF+\DC_\delta$, $\p\in V$ is a forcing notion admitting a gap at $\delta$, and $G\subseteq\p$ is $V$-generic. Then in $V[G]$, the ground model $V$ is definable from the parameter $P(\delta)^V$.
\end{main1}
\begin{proof}
First, observe that we can assume without loss of generality that $\delta$ is regular. If $\delta$ was singular, then we could replace it by the regular cardinal $\delta'=\gamma^+$, where $\p$ factors as $\R*\dot\q$ with $|\R|=\gamma<\delta$, witnessing that it admits a gap at $\delta$.

By Theorem~\ref{th:forcingcoverandapproxzf}, the pair $V\subseteq V[G]$ has the $\delta$-cover and $\delta$-approximation properties. Moreover, as we observed earlier the pair $V\subseteq V[G]$ has the $\delta^+$-cover and $\delta^+$-approximation properties. It follows that any set in $V$ that has size $\lesseq\delta$ in $V[G]$ also has size $\lesseq\delta$ in $V$, and so in particular we have $\her{\delta}^{V[G]}\cap V=\her{\delta}^V$. Finally, by Theorem~\ref{th:dcpreservation}, the forcing extension $V[G]\models\DC_\delta$.

It is easy to see that the $\delta$-cover and $\delta$-approximation properties reflect down to pairs $V_\lambda\subseteq V_\lambda[G]$ for $\lambda$ of cofinality  $\geq\delta$. Moreover, any such $V_\lambda\models\Z+\DC_\delta$.
Thus, the sets $V_\lambda$, for ordinals $\lambda>\delta^+$ of cofinality $\geq\delta$ in $V[G]$, are now defined in $V[G]$ as the unique transitive models $M\models \Z+\DC_\delta$ of height $\lambda$, having $P(\delta)^M=P(\delta)^V$ such that the pair $M\subseteq V[G]_\lambda$ has the $\delta$-cover and $\delta$-approximation properties.
\end{proof}

\section{Undefinable $\ZFC^-$ ground models}\label{sec:zfcminus}
In this section, we show that it is consistent for Laver's and Woodin's ground model definability result to fail for the canonical $\ZFC^-$ models $\her{\kappa}$. Starting with a universe $V$ and a cardinal $\kappa\in V$, we produce a forcing extension $V[G]$ in which $\her{\kappa}^{V[G]}$ is not definable in its forcing extension by a poset of the form $\Add(\delta,1)$. The same forcing construction carried out over the Mostowski collapse of a countable elementary submodel of a sufficiently large $\her{\theta}$ shows the existence of countable transitive models of $\ZFC^-$ that violate ground model definability. All our counterexample models violate ground model definability  in a strong sense, because they have set-forcing extensions in which they are not definable, not even when using parameters from the extension.

%In what follows, we will view $\Add(\delta,\gamma)$ as product of $\gamma$-many posets, each adding a subset to $\delta$ using conditions of size less than $\delta$, with support of size less than $\delta$.

\begin{main2}
Assume that $\delta, \kappa$ are cardinals such that $\delta$ is regular and either $2^{\lt\delta}\!<\!\kappa$\, or \,$\delta\!=\!\kappa\!=\!2^{\lt\kappa}$ holds. If $V[G]$ is a forcing extension by $\Add(\delta,\kappa^+)$,
%\emph{(}with $2^{\lt\delta}<\kappa$ or $\delta=\kappa$ inaccessible\emph{)}
then $\her{\kappa}^{V[G]}$ is not definable in its forcing extension by $\Add(\delta,1)$. It follows that there exists a countable transitive model of $\ZFC^-$ that is not definable in its Cohen forcing extension.
\end{main2}
\begin{proof} For ease of presentation we shall only prove the specific case of the theorem when $\delta=\omega$ and $\kappa=\omega_1$, so that $V[G]$ is the forcing extension by $\Add(\omega,\omega_2)$. We shall say a few words about the proof of the general case at the end of this proof.

Fix any $V[G]$-generic $g\subseteq \Add(\omega,1)$. We will show that $H_{\omega_2}^{V[G]}$ is not definable in its $\Add(\omega,1)$-forcing extension $H_{\omega_2}^{V[G]}[g]$, not even when using parameters from the forcing extension $H_{\omega_2}^{V[G]}[g]$. In $V[G]$, every nice $\Add(\omega,1)$-name for a subset of $\omega_1$ has hereditary size at most $\omega_1$, and since every element of $H_{\omega_2}^{V[G]}$ can be coded by a subset of $\omega_1$ via the Mostowski collapse, it follows that $H_{\omega_2}^{V[G][g]}= H_{\omega_2}^{V[G]}[g]$. Using this equality, we now suppose towards a contradiction that $H_{\omega_2}^{V[G]}$ is definable in $H_{\omega_2}^{V[G][g]}$ by some formula $\varphi(x,a)$ for some parameter $a\in H_{\omega_2}^{V[G][g]}$. Without loss of generality, assume that $a\subseteq\omega_1$. Let $\dot a$ be a nice $\Add(\omega,\omega_2)\times \Add(\omega,1)$-name for a subset of $\omega_1$ such that $(\dot a)_{G\times g}=a$, and let $\dot G$ be the canonical $\Add(\omega,\omega_2)\times \Add(\omega,1)$-name for a generic filter on the $\Add(\omega,\omega_2)$ part of the product. Let $g_\alpha$ denote the Cohen subset on coordinate $\alpha$ of $G$. Since each $g_\alpha$ is an element of $H_{\omega_2}^{V[G]}$, we have that $H_{\omega_2}^{V[G][g]}\models \varphi(g_\alpha,a)$ for all $\alpha<\omega_2$. Thus, we proceed to fix a condition $(p,q)\in G\times g$ forcing that for all $\alpha<\omega_2$, $H_{\omega_2}\models\varphi(x,\dot a)$ of the Cohen subset $x$ on coordinate $\alpha$ of $\dot G$. Since $\Add(\omega,\omega_2)$ has the ccc, there is $\beta<\omega_2$ such that $(\dot a)_{G\times g}=(\dot a)_{G_\beta\times g}$, where $G_\beta$ is the restriction of $G$ to the first $\beta$-many coordinates of $\Add(\omega,\omega_2)$. We may choose $\beta$ large enough so that $p\in G_\beta$.

Fix any ordinal $\gamma$ in $\omega_2$ above $\beta$. A standard approach to take towards obtaining a contradiction would be to try to interchange $g$ with $g_\gamma$. This doesn't quite work because $a$ is an element of the extension $H_{\omega_2}^{V[G][g]}$ and therefore $g$ may be necessary to interpret $\dot a$ correctly. Instead, our strategy will be to use an automorphism $\pi$ of $\Add(\omega,1)$ in $V[g]$ whose point-wise image of $g_\gamma$ produces a filter $\overline g_\gamma$ that together with $g_\gamma$ codes in $g$. For instance, we can take $\pi$ to be the function mapping a sequence $s\in \Add(\omega,1)$ to the sequence of the same length with the value of a bit of $s$ flipped whenever the value of that bit in $g$ is $1$. Viewing $\pi$ as an automorphism of $\Add(\omega,\omega_2)$ which acts only on coordinate $\gamma$, we obtain, by applying it to $G$, the filter $\overline G$, where $g_\gamma$ is replaced by $\overline g_\gamma$ on coordinate $\gamma$. Since $\pi$ is an automorphism of $\Add(\omega,\omega_2)$ in $V[g]$, the filter $\overline G$ is $V[g]$-generic and $V[g][G]=V[g][\overline G]$. Because we are using product forcing, it follows that $V[G][g]=V[\overline G][g]$. The condition $(p,q)$ is an element of $\overline G\times g$, since $G_\beta$ is an initial segment of $\overline G$, and it forces that $H_{\omega_2}\models\varphi(x,\dot a)$ of the Cohen subset $x$ on coordinate $\gamma$ of $\dot G$. Since $(\dot G)_{\overline G\times g}=\overline G$ and $(\dot a)_{\overline G\times g}=a$, it follows that $H_{\omega_2}^{V[\overline G][g]}=H_{\omega_2}^{V[G][g]}\models\varphi(\overline g_\gamma,a)$, and hence $\overline g_\gamma$ is an element of $H_{\omega_2}^{V[G]}$. Now we have $g_\gamma$ and $\overline g_\gamma$ both in $H_{\omega_2}^{V[G]}$, from which it follows by the definition of $\pi$ that the filter $g$ is in $H_{\omega_2}^{V[G]}$ as well.  Thus, we have reached a contradiction, showing that $H_{\omega_2}^{V[G]}$ could not have been definable in $H_{\omega_2}^{V[G][g]}$.

To see that there exists a transitive model of $\ZFC^-$ that is not definable in its Cohen extension, we carry out the above forcing construction over the collapse of a countable elementary submodel of some sufficiently large $\her{\theta}$. For instance, suppose $\theta$ is large enough that $H_{\omega_2}\in\her{\theta}$. Let $M$ be a transitive model of $\ZFC^-$ that is the collapse of some countable elementary submodel of $\her{\theta}$. Since $M$ is countable, there is, in $V$, an $M$-generic filter for $\Add(\omega,\omega_2)^M$. Consider the countable transitive $\ZFC^-$-model $N=H_{\omega_2}^{M[G]}$ and use the above argument to show that $N$ is not definable in its Cohen forcing extension.

The proof of the general case for arbitrary cardinals $\delta, \kappa$ is essentially the same. In the case when $2^{\lt\delta}<\kappa$, the poset $\Add(\delta,\kappa^+)$ has the $\kappa$-cc and is $\lt\delta$-closed, and it thus preserves all cardinals ${\leq}\delta$ and all cardinals ${\geq}\kappa$. When $\delta=\kappa=2^{\lt\kappa}$, the poset $\Add(\delta,\kappa^+)$ has the $\kappa^\plus$-cc and is $\lt\kappa$-closed, and it thus preserves all cardinals. In both cases, $\kappa$ is preserved as a cardinal in $V[G]$, and so it makes sense to consider $\her{\kappa}^{V[G]}$. The $\kappa^\plus$-cc allows for the analogous nice name argument.
\end{proof}
\noindent In particular, it follows from the theorem that it is consistent to have cardinals $\kappa>\!>\omega$ such that $\her{\kappa}$ is not definable in its Cohen forcing extension.

Recall from the introduction that following Laver's proof~\cite{laver:groundmodel} a natural parameter to use when defining $V$ in a forcing extension $V[G]$ by a poset $\p$ is the parameter  $P(\delta)^V$, where $\delta=|\p|^+$ in $V$ (see also~\cite{reitz:groundaxiom}). The second author and Joel Hamkins observed in 2012 how this parameter can easily be reduced to the parameter $P(|\p|)^V$, the parameter that Woodin had used in his $\ZFC$ ground model definition~\cite{woodin:groundmodel}.  Namely, if $V\subseteq V[G]$ has the $\delta$-approximation property, then $A\subseteq\delta$ is an element of $P(\delta)^V$ if and only if for every bounded subset $a\in P_\delta(\delta)^V$ the intersection $a\cap A$ is an element of $P_\delta(\delta)^V$.\footnote{As is standard, we let $P_\alpha(\delta)$ denote the collection of all subsets of $\delta$ of size less than $\alpha$.} Thus, $P(\delta)^V$ is definable in $V[G]$ from $P_\delta(\delta)^V$.  But since every element of $P_\delta(\delta)^V$ is coded by an element of $P(|\p|)^V$, it follows that $P(\delta)^V$ is definable in $V[G]$ from $P(|\p|)^V$.
\begin{theorem}
It is consistent that a ground model $V$ cannot be defined in its Cohen forcing extension $V[g]$  with a parameter of hereditary size less than $2^\omega$.
\end{theorem}
\begin{proof}
Suppose that ${\rm CH}$ holds in $V$ and $V[G]$ is a forcing extension by $\Add(\omega,\omega_1)$. Clearly ${\rm CH}$ continues to hold in $V[G]$. Let $V[G][g]$ be the forcing extension of $V[G]$ by $\Add(\omega,1)$. The proof of  Main~Theorem~2 easily modifies to show that $P(\omega)^{V[G]}$, and thus, $V[G]$ itself, cannot be defined in $V[G][g]$ using a parameter from $H_{\omega_1}$.
\end{proof}

The pair $H_{\omega_2}^{V[G]}$ and $H_{\omega_2}^{V[G][g]}$, from the proof of Main~Theorem~2, witnesses another violation of a standard $\ZFC$-result about ground models and their forcing extensions. A model of $\ZFC$ can never be an elementary submodel of its set-forcing extension since, using the ground model powerset of the poset as a parameter, it is expressible that the forcing extension contains a filter meeting all the ground model dense sets. But in set theory without powerset, forcing extensions can be elementary.
\begin{theorem}[Hamkins \cite{hamkins:elementaryforcingextension}]
It is consistent that there is a model $M=\her{\kappa}$ and a poset $\p\in M$ such that $M$ is an elementary submodel of its forcing extensions by $\p$.
\end{theorem}
\begin{proof}
For instance, consider the models $M=H_{\omega_2}^{V[G]}$ and $M[g]=H_{\omega_2}^{V[G][g]}$ from the proof of Main~Theorem~2. We  argue that $M\prec M[g]$. Suppose $M\models \varphi(a)$ and assume that \hbox{$a\subseteq\omega_1$}. Fix an $\Add(\omega,\omega_2)$-name $\dot a$ such that $(\dot a)_G=a$ and a condition $p\in \Add(\omega,\omega_2)$ forcing that $H_{\check{\omega}_2}\models \varphi(\dot a)$. Now observe that the poset $\Add(\omega,\omega_2)\times \Add(\omega,1)$ is isomorphic to $\Add(\omega,\omega_2)$, and such an isomorphism may be chosen to fix any initial segment of the product $\Add(\omega,\omega_2)$. Thus, we let $F:\Add(\omega,\omega_2)\times \Add(\omega,1)\to \Add(\omega,\omega_2)$ be an isomorphism that fixes $p$ and the name $\dot a$. Next, we let $\overline G$ be the image of $G\times g$ under $F$ and observe that $V[G][g]=V[\overline G]$ and $p\in \overline G$. Since $p\in\overline G$ and $(\dot a)_{\overline G}=(\dot a)_G$, it follows that $H_{\omega_2}^{V[\overline G]}=H_{\omega_2}^{V[G][g]}=M[g]$ satisfies $\varphi(a)$. Since we chose an arbitrary formula $\varphi(x)$ with $a$ an arbitrary element of $M$, this concludes the proof that $M\prec M[g]$.
\end{proof}
All our models violating ground model definability share the feature that the powerset of the poset in whose forcing extension they are not definable is too large to be an element of the model. Indeed, it is by exploiting this very feature that the counterexample models are obtained. Is it possible that a model of  $\ZFC^-$ is not definable in its forcing extension by a poset whose powerset is an element of the model? Is it possible that a model $\her{\kappa}$ is not definable in its forcing extension by a poset whose powerset has size $\lesseq\kappa$?
David Asper\'{o} communicated to the authors that, by a result of Woodin, such a model exists in a universe with an $I_0$-cardinal, one of the strongest known large cardinal notions. A cardinal $\kappa<\lambda$ is an $I_0$-cardinal if it is the critical point of an elementary embedding $j:L(V_{\lambda+1})\to L(V_{\lambda+1})$. The $I_0$-cardinals were introduced by Woodin, and their existence pushes right up against the \emph{Kunen Inconsistency}, the existence of a nontrivial  elementary \hbox{embedding $j:V\to V$.}
\begin{theorem}[Woodin \cite{woodin:suitableextendermodels2}]\label{th:zfcminusundefinablesmall}
If there is an elementary embedding\break $j:L(V_{\lambda+1})\to L(V_{\lambda+1})$ with critical point $\kappa<\lambda$, then $\her{\lambda}$ is not definable in its forcing extension by any poset $\p\in V_\lambda$ adding a countable sequence of elements of $\lambda$. In particular, $\her{\lambda}$ is not definable in its Cohen extensions.
\end{theorem}
\begin{proof}
Since $\p\in V_{\lambda}$, we have that $V_{\lambda+1}^{V[G]}=V_{\lambda+1}[G]$ and $\her{\lambda}^{V[G]}=\her{\lambda}[G]$. Thus, whenever $\her{\lambda}$ is definable in $\her{\lambda}^{V[G]}\in L(V_{\lambda+1}^{V[G]})$, we must have $V_{\lambda+1}\in L(V_{\lambda+1}^{V[G]})$. Indeed, in this case, $V_{\lambda+1}$ already by appears by some finite stage $L_n(V_{\lambda+1}^{V[G]})$ of the construction because $H_{\lambda^+}^{V[G]}$ is isomorphic to the structure $\mathcal H$ built out of equivalence classes of subsets of $\lambda$ coding extensional well-founded relations. The equivalence relation, as well as the membership relation, on the codes is definable from information in $V_{\lambda+1}^{V[G]}$. Using the definition of $H_{\lambda^+}$ in $\her{\lambda}^{V[G]}$, we can recover elements of $V_{\lambda+1}$ from $\mathcal H$ because their Mostowski collapses already exist in $V_{\lambda+1}^{V[G]}$. The theorem now follows directly from a result of Woodin showing that if $j:L(V_{\lambda+1})\to L(V_{\lambda+1})$ is an elementary embedding with critical point $\kappa<\lambda$, $\p\in V_\lambda$ is a poset, and $(^\omega\lambda)^V\neq (^\omega\lambda)^{V[G]}$ in a forcing extension $V[G]$ by $\p$, then $V_{\lambda+1}\notin L_\lambda(V_{\lambda+1}^{V[G]})$.
\end{proof}
\noindent It appears that nothing else is currently known about whether large cardinals are needed for the existence of such counterexample models.

The next observation is intended as a road map of which paths to avoid when attempting to construct a counterexample model $\her{\kappa}$ that is not definable in its forcing extension by some poset whose powerset has size $\lesseq\kappa$.
\begin{theorem}\label{th:sufficient}
Suppose $\p\in\her{\kappa}$ is a forcing notion of size $\gamma$ and $\kappa^\gamma=\kappa$. Then $\her{\kappa}$ is definable in its forcing extensions by $\p$ using the ground model parameter $P_{\gamma^+}(\kappa)$. For instance, we have:
\begin{itemize}
\item[(1)] If the \GCH\ holds and $\p\in\her{\kappa}$ with $|\p|<\text{\emph{cf}}(\kappa)$, then $\her{\kappa}$ is definable in all its forcing extensions by $\p$. In particular, for $\kappa$ such that $\text{cf}(\kappa)>\omega$, we then have that $\her{\kappa}$ is definable in its Cohen extensions.
\item[(2)] If $\kappa^{\lt\kappa}=\kappa$, then $\her{\kappa}$ is definable in all its forcing extensions by posets of size less than $\kappa$.
\end{itemize}
\end{theorem}
\begin{proof}
First, observe that if $\p$ is a poset of size $\gamma$ in a model $M\models\ZFC^-$ and $\gamma$ is not the largest cardinal of $M$, then the usual $\ZFC$ argument generalizes easily to show that the pair $M\subseteq M[G]$ has the $\delta$-cover and $\delta$-approximation properties for $\delta=\gamma^+$. Now let $G\subseteq\p$ be $\her{\kappa}$-generic for a poset as in the statement of the theorem. To see that $\her{\kappa}$ is definable in $\her{\kappa}[G]$, it suffices to definably select those subsets of $\kappa$ that belong to $\her{\kappa}$. Thus, consider in $\her{\kappa}$ the definable class $b=P_{\gamma^+}(\kappa)$. Since $b$ has size $\kappa$ by assumption, $b$ exists as a set in $\her{\kappa}$. Since $\her{\kappa}\subseteq\her{\kappa}[G]$ has the $\delta$-approximation property, it follows that in $\her{\kappa}[G]$ a set $A\subseteq\kappa$ is in $\her{\kappa}$ if and only if for every $a\in b$ the set $A\cap a$ is an element of $b$. Assertions~(1)~and~(2) are immediate consequences.
\end{proof}
\noindent Note that even if $P_{\gamma^+}(\kappa)$ is too large to be an element of $\her{\kappa}$, it would suffice for the proof if it was definable in the forcing extension $\her{\kappa}[G]$. Thus, for example, it is relatively consistent that $\her{\aleph_\omega}$ is definable in its Cohen extensions, even if $2^{\omega_1}>\!>\aleph_\omega$, by starting in $L$ and blowing up the powerset of $\omega_1$. Note also that since any $\kappa$ with $\kappa^{\lt\kappa}=\kappa$ retains this property after forcing with $\Add(\kappa,\kappa^+)$, it is relatively consistent by Main~Theorem~2 that $\her{\kappa}$, for such a cardinal $\kappa$, is not definable in its forcing extension by $\Add(\kappa,1)$. Thus, the size requirement on the posets in assertion~(2) of Theorem~\ref{th:sufficient} is optimal.

The proof of Theorem~\ref{th:sufficient} applies to arbitrary transitive $\ZFC^-$-models with a  largest cardinal also. Indeed, if $M\models\ZFC^-$ with the largest cardinal $\kappa\in M$, and $\p\in M$ is a forcing notion of size $\gamma$ with $\kappa^\gamma=\kappa$ in $M$, then $M$ is definable in its forcing extensions by $\p$ using the parameter $P_{\gamma^+}(\kappa)^M$. If in addition $\kappa^{\lt\kappa}=\kappa$ holds in $M$, then $M$ is definable in all its forcing extensions by posets of size less than $\kappa$, using the same parameter $P_{\gamma^+}(\kappa)^M$. However, $M$ need not be definable in its forcing extensions by posets of size $\kappa$. Several large cardinal notions $\kappa$ below a measurable cardinal are characterized by the existence of elementary embeddings of such $\ZFC^-$-models with $\kappa$ as the largest cardinal and the critical point of the embedding. Consistency results concerning these large cardinals are obtained by forcing over such models and so it should be noted that the observations above completely characterize the ground definability situation for such models.

\section{Questions}\label{sec:questions}
There remain several open questions surrounding the topics of this paper. In Main~Theorem~1, we generalized ground model definability to models of $\ZF+\DC_\delta$ and posets admitting a gap at $\delta$. We conjecture that in general a model of $\ZF$ need not be definable in its set-forcing extension.
\begin{question}
Is every model of $\ZF$ a definable class of its set-forcing extensions?
\end{question}
\noindent More specifically, since it follows from Main~Theorem~1 that every model of\break $\ZF+\DC_{\omega_1}$ is definable its Cohen forcing extensions, we ask:
\begin{question}
Is every model $V\models \ZF+\DC_\omega$ definable in its Cohen extensions?
\end{question}
The class of posets admitting a closure point at $\delta$, defined in Section~\ref{sec:preliminaries}, is a natural extension of the class of posets admitting a gap at $\delta$. What additional assumptions must be added to generalize  Main~Theorem~1  to posets admitting a closure point at $\delta$? By Theorem~\ref{th:forcingcoverandapproxzf}, the pair consisting of a $\ZF+\DC_\delta$ ground model and its forcing extension by a poset admitting a closure point at $\delta$ has the $\delta^+$-cover and $\delta^+$-approximation properties. Since Theorem~\ref{th:coverandapproxzf} requires $\DC_{\delta^+}$ to conclude uniqueness for submodels with the $\delta^+$-cover and $\delta^+$-approximation properties, it seems reasonable to expect that the necessary choice principle will have to be strengthened to $\DC_{\delta^+}$.
\begin{question}\label{q3}
Is every model $V\models \ZF+\DC_{\delta^+}$ definable in its set-forcing extensions by posets admitting a closure point at $\delta$?
\end{question}
\noindent There are two difficulties in answering this question using our methods. First, we do not know whether posets admitting a closure point at $\delta$ preserve $\DC_{\delta^+}$ (by Theorem~\ref{th:dcpreservation}, we know only that they preserve $\DC_\delta$). Second, forcing extensions by posets admitting a closure point at $\delta$ may collapse $\delta^{\plusplus}$, violating the $H_{\delta^{\plusplus}}^{V[G]}\cap V=H_{\delta^{\plusplus}}^V$ requirement, a condition that was crucially used in the proof of the uniqueness Theorem~\ref{th:coverandapproxzf}.
% In fact, the later difficulty is also encountered in the $\ZFC$ context in trying to extend ground model definability to class posets admitting a closure point at $\delta$.
A resolution of question~\ref{q3} may come down to answering the following.
\begin{question}
Is the $\her{\delta}^W\cap V=\her{\delta}^V$ requirement necessary in Theorem~\ref{th:coverandapproxzf}\,?
\end{question}
A natural approach to answer this question may be to first address the analogous situation in the $\ZFC$ context and find out whether the $(\delta^\plus)^V=(\delta^\plus)^W$ requirement is necessary in the uniqueness Theorem~\ref{th:coverandapprox}.
In regards to the preservation of the choice fragments $\DC_\delta$ by forcing, it seems natural to ask whether Theorem~\ref{th:DCPreservation2} can be strengthened to show that $\DC_\delta$ is preserved by $\leq\delta$-distributive forcing.
\begin{question}
Do $\lesseq\delta$-distributive posets preserve $\DC_\delta$ to the forcing extension?
\end{question}

In Main~Theorem~2, we showed that there always exist models of $\ZFC^-$ violating ground model definability, and also that it is consistent for this to be the case for canonical models $\her{\kappa}$. But all our counterexample models had the feature that the powerset of the forcing poset was a proper class. By a result of Woodin, we know that in a universe with an $I_0$-cardinal, there is a model $\her{\lambda}$ that is not definable in any of its forcing extensions by a poset of size less than $\lambda$ that adds a countable sequence of elements of $\lambda$ (Theorem~\ref{th:zfcminusundefinablesmall}). This introduces the exciting possibility that the existence of $\ZFC^-$-models violating ground model definability for such posets may carry large cardinal strength.
\begin{question}
What is the consistency strength of the existence of a model $\her{\kappa}$ that is not definable in its forcing extension by a poset whose powerset has size $\lesseq\kappa$?
\end{question}
\noindent A more targeted approach would be to settle the situation with the definability of models $\her{\kappa}$ in their Cohen extensions. In this case, Theorem~\ref{th:sufficient}~(1) suggests that such a $\kappa$ should be a singular cardinal $\kappa$ of cofinality $\omega$. Thus, it is not coincidental that Woodin's $\lambda$ is just such a cardinal.
\begin{question}
If $2^\omega\leq\kappa$, what is the consistency strength of having a model $\her{\kappa}$ that is not definable in its Cohen extension?
\end{question}

\bibliographystyle{alpha}
\bibliography{groundmodels}
\end{document}